\documentclass[12pt]{article}
\usepackage{amsthm,amsfonts,amsmath,amscd,amssymb,latexsym}
\usepackage{mathrsfs}
\usepackage{epsfig,graphicx,color}
\usepackage[all]{xy}
\usepackage{makeidx}
\usepackage{hyperref}
\title{Remarks on the Donaldson metric}

\author{
Robin~S.~Krom\thanks{Partially supported by the
Swiss National Science Foundation Grant 200021-127136}
\\
ETH Z\"urich
}

\date{18 January 2015}

\newtheorem{PARA}{}
\newtheorem{theorem}[PARA]{Theorem}
\newtheorem{corollary}[PARA]{Corollary}
\newtheorem{lemma}[PARA]{Lemma}
\newtheorem{proposition}[PARA]{Proposition}
\newtheorem{definition}[PARA]{Definition}
\newtheorem{remark}[PARA]{Remark}

\newcommand{\MAT}[1]{\left[\begin{array}{#1}}
\newcommand{\RIX}{\end{array}\right]}
\newcommand{\p}{\partial}

\newcommand{\R}{{\mathbb R}}

\newcommand{\cH}{{\mathcal H}}

\newcommand{\cL}{{\mathcal L}}

\newcommand{\sE}{\mathscr{E}}

\newcommand{\sH}{\mathscr{H}}

\newcommand{\sS}{\mathscr{S}}

\newcommand{\Om}{{\Omega}}
\newcommand{\om}{{\omega}}

\renewcommand{\phi}{{\varphi}}

\newcommand{\Hhat}{{\widehat{H}}}
\newcommand{\Khat}{{\widehat{K}}}
\newcommand{\rhohat}{{\widehat{\rho}}}

\newcommand{\sigmahat}{{\widehat{\sigma}}}
\newcommand{\tauhat}{{\widehat{\tau}}}

\newcommand{\Xhat}{{\widehat{X}}}
\newcommand{\Zhat}{{\widehat{Z}}}
\newcommand{\Yhat}{{\widehat{Y}}}
\newcommand{\Thetahat}{{\widehat{\Theta}}}

\newcommand{\im}{{\mathrm{im}}}

\newcommand{\grad}{{\mathrm{grad}}}

\newcommand{\Vect}{{\mathrm{Vect}}}

\newcommand{\dvol}{{\rm dvol}}

\newcommand{\inner}[2]{\langle #1, #2\rangle}
\newcommand{\INNER}[2]{\left\langle #1, #2\right\rangle}

\def\NABLA#1{{\mathop{\nabla\kern-.5ex\lower1ex\hbox{$#1$}}}}
\def\Nabla#1{\nabla\kern-.5ex{}_{#1}}
\def\abs#1{\mathopen|#1\mathclose|}
\def\Abs#1{\left|#1\right|}

\renewcommand{\p}{{\partial}}

\begin{document}
\maketitle

\begin{abstract}
  The Donaldson metric is a metric on the space of symplectic two-forms in a
  fixed cohomology class. It was introduced in~\cite{KROMSAL}. We compute the
  associated Levi-Civita connection, describe it's geodesics and compute the
  formula for the covariant Hessian of an energy functional on the space of
  symplectic structures in a fixed cohomology class, introduced by
  S.~Donaldson in~\cite{DON1}.
\end{abstract}

%%%%%%%%%%%%%%%%%%%%%%%%%%%%%%%%%%%%%%%%%%%%%%%%%%
%%%%%%%%%%%%%%%%%%%%%%%%%%%%%%%%%%%%%%%%%%%%%%%%%%
%%%%%%%%%%%%%%%%%%%%%%%%%%%%%%%%%%%%%%%%%%%%%%%%%%
%%%%%%%%%%%%%%%%%%%%%%  SECTION 1  %%%%%%%%%%%%%%%
%%%%%%%%%%%%%%%%%%%%%%%%%%%%%%%%%%%%%%%%%%%%%%%%%%
%%%%%%%%%%%%%%%%%%%%%%%%%%%%%%%%%%%%%%%%%%%%%%%%%%
%%%%%%%%%%%%%%%%%%%%%%%%%%%%%%%%%%%%%%%%%%%%%%%%%%

%\section{Introduction}\label{sec:INTRO}

%%%%%%%%%%%%%%%%%%%%%%%%%%%%%%%%%%%%%%%%%%%%%%%%%%
%%%%%%%%%%%%%%%%%%%%%%%%%%%%%%%%%%%%%%%%%%%%%%%%%%
%%%%%%%%%%%%%%%%%%%%%%%%%%%%%%%%%%%%%%%%%%%%%%%%%%
%%%%%%%%%%%%%%%%%%%%%%  SECTION 2  %%%%%%%%%%%%%%%
%%%%%%%%%%%%%%%%%%%%%%%%%%%%%%%%%%%%%%%%%%%%%%%%%%
%%%%%%%%%%%%%%%%%%%%%%%%%%%%%%%%%%%%%%%%%%%%%%%%%%
%%%%%%%%%%%%%%%%%%%%%%%%%%%%%%%%%%%%%%%%%%%%%%%%%%

%%%%%%%%%%%%%%%%%%%%%%%%%%%%%%%%%%%%%%%%%%%%%%%%%%
%%%%%%%%%%%%%%%%%%%%%%%%%%%%%%%%%%%%%%%%%%%%%%%%%%
%%%%%%%%%%%%%%%%%%%%%%%%%%%%%%%%%%%%%%%%%%%%%%%%%%
%%%%%%%%%%%%%%%%%%%%%%  SECTION 3  %%%%%%%%%%%%%%%
%%%%%%%%%%%%%%%%%%%%%%%%%%%%%%%%%%%%%%%%%%%%%%%%%%
%%%%%%%%%%%%%%%%%%%%%%%%%%%%%%%%%%%%%%%%%%%%%%%%%%
%%%%%%%%%%%%%%%%%%%%%%%%%%%%%%%%%%%%%%%%%%%%%%%%%%

Let $M$ be a closed oriented Riemannian four-manifold.
Denote by $g$ the Riemannian metric on $M$,
denote by ${\dvol\in\Om^4(M)}$ the volume form of $g$,
and let $*:\Om^k(M)\to\Om^{4-k}(M)$ be the
Hodge $*$-operator associated to the metric and orientation.
Fix a cohomology class $a\in H^2(M;\R)$
such that $a^2>0$ and consider the space
$$
\sS_a:=\left\{\rho\in\Om^2(M)\,\big|\,
d\rho=0,\,\rho\wedge\rho>0,\,[\rho]=a\right\}
$$
of symplectic forms on $M$ representing the class $a$.
This is an infinite-dimensional manifold and the
tangent space of~$\sS_a$ at any element~${\rho\in\sS_a}$
is the space of exact $2$-forms on $M$.
The next proposition is proved in~\cite{KROMSAL}. It summarizes the properties
of a family of Riemannian metrics $g^\rho$ on~$M$, one for each nondegenerate
$2$-form $\rho$ (and for each fixed background metric $g$). For each
nondegenerate $2$-form define the function $u$ by the equation
\begin{equation}\label{eq:u}
  2 u\dvol = \rho\wedge\rho
\end{equation}

\begin{proposition}[{\bf Symplectic Forms and Riemannian Metrics}]
\label{prop:grho}\

\noindent
Fix a nondegenerate $2$-form $\rho\in\Om^2(M)$
such that $\rho\wedge\rho>0$ and define the
function $u:M\to(0,\infty)$ by~\eqref{eq:u}.
Then there exists a unique Riemannian metric
$g^\rho$ on $M$ that satisfies the following conditions.

\smallskip\noindent{\bf (i)}
The volume form of $g^\rho$ agrees with the volume form of $g$.

\smallskip\noindent{\bf (ii)}
The Hodge $*$-operator $*^\rho:\Om^1(M)\to\Om^3(M)$
associated to $g^\rho$ is given
by
\begin{equation}\label{eq:*rhola}
*^\rho\lambda = \frac{\rho\wedge*(\rho\wedge\lambda)}{u}
\end{equation}
for $\lambda\in\Om^1(M)$ and by
$*^\rho\iota(X)\rho = -\rho\wedge g(X,\cdot)$
for $X\in\Vect(M)$.

\smallskip\noindent{\bf (iii)}
The Hodge $*$-operator $*^\rho:\Om^2(M)\to\Om^2(M)$
associated to $g^\rho$ is given by
\begin{equation}\label{eq:*rhom}
*^\rho\om = R^\rho*R^\rho\om,\qquad
R^\rho\om := \om-\frac{\om\wedge\rho}{\dvol_\rho}\rho,
\end{equation}
for $\om\in\Om^2(M)$. The linear map $R^\rho:\Om^2(M)\to\Om^2(M)$
is an involution that preserves the exterior product,
acts as the identity on the orthogonal complement
of $\rho$ with respect to the exterior product,
and sends $\rho$ to $-\rho$.

\smallskip\noindent{\bf (iv)}
Let $\om\in\Om^2(M)$ be a nondegenerate $2$-form and let
$J:TM\to TM$ be an almost complex structure such that
$g = \om(\cdot,J\cdot)$.  Define the almost complex
structure $J^\rho$ by $\rho(J^\rho\cdot,\cdot):=\rho(\cdot,J\cdot)$
and define the $2$-form $\om^\rho\in\Om^2(M)$ by $\om^\rho:=R^\rho\om$.
Then $g^\rho = \om^\rho(\cdot,J^\rho\cdot)$ and so
$\om^\rho$ is self-dual with respect to $g^\rho$.
\end{proposition}

The following is the central object of this paper.
\begin{definition}\label{def:DONmet}
Each nondegenerate $2$-form $\rho\in\Om^2(M)$ with $\rho^2>0$
determines an inner product $\inner{\cdot}{\cdot}_\rho$ on the space
of exact $2$-forms defined by
\begin{equation}\label{eq:innerho}
\INNER{\rhohat_1}{\rhohat_2}_\rho
:= \int_M\lambda_1\wedge *^\rho\lambda_2,\qquad
d\lambda_i=\rhohat_i,\qquad *^\rho\lambda_i\in\im\,d.
\end{equation}
These inner products determine a Riemannian metric on the
infinite-di\-men\-si\-onal manifold $\sS_a$ called
the {\bf Donaldson metric}.
\end{definition}
\begin{definition}
A vector field $X_\rhohat$ is {\bf associated} to an exact 2-form $\rhohat \in
T_\rho \sS_a$ if it is the unique vector field satisfing
  \begin{equation}\label{eq:associatedVecField}
    - d \iota(X_\rhohat) \rho = \rhohat, \qquad *^\rho \iota(X_\rhohat) \rho \in \im d.
  \end{equation}
\end{definition}

Every metric has an associated Levi-Civita connection. This is the unique
torsion free and Riemannian connection with respect to the metric. The formula
and computation is the content of the following theorem.
\begin{theorem}[{\bf Levi Civita Connection}]\label{thm:LeviCivita}
  Let $\rho_t:\R \to \sS_a$ be a smooth path of symplectic forms with $\rho :=
  \rho_0$ and $\rhohat := \left.\p_t\right|_{t=0}\rho_t$. Let $X$ be the associated
  vector field of $\rhohat$. Let $Y_t: \R \to \Vect(M)$ be a smooth path of
  vector fields such that $*^\rho \iota(Y_t) \rho_t$ is exact and define
  \begin{equation*}
    \sigmahat_t := -d \iota(Y_t) \rho_t, \quad \sigmahat := \sigmahat_0,
    \quad Y := Y_0.
  \end{equation*}
  The unique Levi-Civita connection associated to the Donaldson metric
  is given by
  \begin{equation}\label{eq:LeviCivita}
    \nabla^\rho_{\rhohat} {\sigmahat} = \left. \frac{d}{dt}\right|_{t=0}
    \sigmahat_t + \frac12d\iota(Y)\rhohat + \frac12
    d\iota(X)\sigmahat - \frac12d\iota\left( \nabla_{Y}X +
    \nabla_{X}Y  \right)\rho.
  \end{equation}
  Here $\nabla_X Y$ denotes the covariant derivative of the
  Levi-Civita connection of the metric $g$ for two vector fields $X$
  and $Y$.
\end{theorem}
\begin{proof}
  The Levi-Civita connection of the Donaldson metric is the unique connection
  that is torsion free and Riemannian with respect to the Donaldson metric.
  Since the Christoffel symbol given by~\eqref{eq:LeviCivita} is symmetric in
  $\rhohat$ and $\sigmahat$, the torsion of the connection $\nabla^\rho$
  vanishes. It remains to show that it is Riemannian. Let $Z_t$ be a smooth path
  of vector fields such that $*^{\rho_t} \iota(Z_t)\rho_t $ is exact. Denote
  $\tauhat_t := -d\iota(Z_t) \rho_t$, $\tauhat := \tauhat_0$. We claim that
  $$
  \left.\frac{d}{dt}\right|_{t=0}\inner{\sigmahat_t}{\tauhat_t}_{\rho_t}
  = \inner{\nabla_\rhohat\sigmahat}{\tauhat}_\rho + \inner{\sigmahat}{\nabla_\rhohat
  \tauhat}_\rho.
  $$
  By the definition of the Donaldson metric and since
  $
  *^\rho \iota(X)\rho = - \rho \wedge \iota(X) g
  $
  by Proposition 3.1 ii) in~\cite{KROMSAL},
  \begin{equation*}
    \inner{\sigmahat_t}{\tauhat_t}_{\rhohat_t} = \int_M
    \left(\iota(Y_t)\rho_t\right) \wedge *^\rho \iota(Z_t)\rho_t = -\int_M
    \left(\iota(Y_t) \rho_t\right) \wedge \rho_t \wedge \iota(Z_t) g
  \end{equation*}
  and
  \begin{equation*}
    \begin{split}
      \inner{\left.\frac{d}{dt}\right|_{t=0}\sigmahat_t}{\tauhat}_\rho
      &=
      -\int_M ( \iota(\Yhat)\rho ) \wedge \rho \wedge \iota(Z)g - \int_M
      ( \iota(Y)\rhohat ) \wedge \rho \wedge \iota(Z)g\\
      \inner{\sigmahat}{\left.\frac{d}{dt}\right|_{t=0}
        \tauhat_t}_\rho
        &=
        -\int_M ( \iota(\Zhat)\rho ) \wedge \rho \wedge \iota(Y)g - \int_M
        ( \iota(Z)\rhohat ) \wedge \rho \wedge \iota(Y)g,
    \end{split}
  \end{equation*}
  where
  $
  \Yhat := \left. \frac{d}{dt}\right|_{t=0} Y_t
  $,
  $
  \Zhat := \left. \frac{d}{dt}\right|_{t=0} Z_t
  $.
  Hence,
  \begin{equation*}
    \begin{split}
      &
      \left. \frac{d}{dt}\right|_{t=0}
      \inner{\sigmahat_t}{\tauhat_t}_{\rhohat_t}\\
      &=
      \inner{\left.\frac{d}{dt}\right|_{t=0}\sigmahat_t}{\tauhat}_\rho -
      \int_M ( \iota(Y) \rho ) \wedge \rhohat \wedge \iota(Z)g - \int_M
      ( \iota(Y)\rho ) \wedge \rho \wedge \iota(\Zhat) g\\
      &=
      \inner{\left.\frac{d}{dt}\right|_{t=0}\sigmahat_t}{\tauhat}_\rho -
      \int_M ( \iota(Y) \rho ) \wedge \rhohat \wedge \iota(Z)g - \int_M
      ( \iota(\Zhat)\rho ) \wedge \rho \wedge \iota(Y) g\\
      &=
      \inner{\left.\frac{d}{dt}\right|_{t=0}\sigmahat_t}{\tauhat}_\rho
      + \inner{\sigmahat}{\left.\frac{d}{dt}\right|_{t=0}\tauhat_t}\\
      &\qquad
      -
      \int_M ( \iota(Y) \rho ) \wedge \rhohat \wedge \iota(Z)g + \int_M
      ( \iota(Z)\rhohat ) \wedge \rho\wedge \iota(Y)g.
    \end{split}
  \end{equation*}
  Define the Christoffel symbols
  $\Gamma_{\sigmahat\tauhat\rhohat}$
  by
  \begin{equation*}
    \begin{split}
      2\Gamma_{\sigmahat\tauhat\rhohat}
      &:=
        ( \iota(Y) \rho ) \wedge ( d\iota(X)\rho ) \wedge \iota(Z)g -
        ( \iota(Z)d\iota(X)\rho ) \wedge \rho\wedge \iota(Y)g\\
      &\qquad
      +( \iota(Y) \rho ) \wedge ( d\iota(Z)\rho ) \wedge \iota(X)g -
        ( \iota(X)d\iota(Z)\rho ) \wedge \rho\wedge \iota(Y)g\\
      &\qquad
      -( \iota(Z) \rho ) \wedge ( d\iota(Y)\rho ) \wedge \iota(X)g +
        ( \iota(X)d\iota(Y)\rho ) \wedge \rho\wedge \iota(Z)g.
    \end{split}
  \end{equation*}
  Then
  \begin{equation*}
    \begin{split}
      \Gamma_{\sigmahat\tauhat\rhohat} + \Gamma_{\tauhat\sigmahat\rhohat}
      &=
      ( \iota(Y) \rho ) \wedge d\iota(X)\rho \wedge \iota(Z)g -
      ( \iota(Z)d\iota(X)\rho ) \wedge \rho\wedge \iota(Y)g\\
      &=
      ( -\iota(Y) \rho ) \wedge \rhohat \wedge \iota(Z)g +
      ( \iota(Z)\rhohat ) \wedge \rho\wedge \iota(Y)g.
    \end{split}
  \end{equation*}
  Hence, it remains to show that
  $$
  \int_M \Gamma_{\sigmahat\tauhat\rhohat} =
  \inner{\sigmahat}{\frac12 d\iota(Z)\rhohat + \frac12
  d\iota(X)\tauhat - \frac12 d\iota(\nabla_Z X + \nabla_X
Z)\rho}_\rho.
  $$
  Let
  \begin{equation*}
    \begin{split}
      A
      &:=
      ( -\iota(Z)d\iota(X)\rho ) \wedge \rho\wedge \iota(Y)g - (
      \iota(X)d\iota(Z)\rho ) \wedge \rho\wedge \iota(Y)g\\
      B
      &:=
      2\Gamma_{\sigmahat \tauhat \rhohat} - A\\
      &=
      ( \iota(Y) \rho ) \wedge ( d\iota(X)\rho ) \wedge \iota(Z)g
      +( \iota(Y) \rho ) \wedge ( d\iota(Z)\rho ) \wedge \iota(X)g\\
      &\qquad
      -( \iota(Z) \rho ) \wedge ( d\iota(Y)\rho ) \wedge \iota(X)g +
        ( \iota(X)d\iota(Y)\rho ) \wedge \rho\wedge \iota(Z)g.
    \end{split}
  \end{equation*}
  Since
  $
  *^\rho \iota(X)\rho = -\rho \wedge \iota(X) g
  $
  for any vector field $X$,
  \begin{equation*}
    \begin{split}
      \int_M A
      &=
      - \int_M \iota(Z)d\iota(X)\rho \wedge \rho \wedge \iota(Y) g - \int_M \iota(X) d\iota(Z) \rho \wedge \rho \wedge \iota(Y) g\\
      &\qquad =
      \inner{\sigmahat}{d\iota(Z)\rhohat + d \iota(X) \tauhat}_\rho.
    \end{split}
  \end{equation*}
  Since
  \begin{multline*}
    Zg(X,Y) + Xg(Y,Z) - Yg(X,Z)\\
    =
    g(\nabla_ZX + \nabla_XZ, Y) + g([Y,Z],X) + g([Y,X],Z)
  \end{multline*}
  (see our sign convention for the Lie bracket), we have
  \begin{equation*}
    \begin{split}
      \inner{\sigmahat}{-d\iota(\nabla_Z X + \nabla_X Z) \rho}_\rho
      &=
      \int_M g(Y, \nabla_Z X + \nabla_X Z) \dvol_\rho\\
      &=
      \int_M \big(Zg(X,Y) + Xg(Y,Z) - Yg(X,Z)\\
      &
      \qquad \qquad - g([Y,Z],X) - g([Y,X],Z)\big) \dvol_\rho.
    \end{split}
  \end{equation*}
  Here we used that
  $
  \inner{-d\iota(X) \rho}{-d\iota(Y) \rho}_\rho = \int_M g(X,Y) \dvol_\rho
  $
  for any vector field $X$ and a vector field $Y$ such that $*^\rho\iota(Y)\rho$
  is exact. Then
  \begin{equation*}
    \begin{split}
      &-\int_M g([Y,Z],X) + g([Y,X],Z)\dvol_\rho\\
      &=
      -\inner{d\iota([Y,Z])\rho}{d\iota(X)\rho}_\rho -
      \inner{d\iota([Y,X])\rho}{d\iota(Z)\rho}_\rho\\
      &=
      -\inner{\cL_{[Y,Z]}\rho}{d\iota(X)\rho}_\rho -
      \inner{\cL_{[Y,X]}\rho}{d\iota(Z)\rho}_\rho\\
      &=
      \inner{[\cL_Y,\cL_Z]\rho}{d\iota(X)\rho}_\rho +
      \inner{[\cL_Y,\cL_X]\rho}{d\iota(Z)\rho}_\rho\\
      &=
      -\int_M ( \iota(Y)d\iota(Z)\rho ) \wedge \rho \wedge \iota(X) g + \int_M (
      \iota(Z)d\iota(Y) \rho ) \wedge \rho \wedge \iota(X) g\\
      &\qquad
      - \int_M ( \iota(Y)d\iota(X)\rho ) \wedge\rho\wedge \iota(Z)g + \int_M (
      \iota(X)d\iota(Y)\rho ) \wedge \rho \wedge \iota(Z) g.
    \end{split}
  \end{equation*}
  Here we used the identity $\cL_{[X,Y]} = - [\cL_X, \cL_Y]$ for all vector
  fields $X,Y$ in the third equality.  Using the Leibniz rule for the interior
  product for the first three terms yields,
  \begin{equation*}
    \begin{split}
      &-\int_M g([Y,Z],X) + g([Y,X],Z)\dvol_\rho\\
      &=
      \int_M ( d\iota(Z)\rho ) \wedge ( \iota(Y)\rho ) \wedge \iota(X) g +\int_M g(X,Y) ( d\iota(Z)\rho ) \wedge \rho\\
      &\qquad
      - \int_M ( d\iota(Y) \rho ) \wedge ( \iota(Z) \rho ) \wedge \iota(X) g- \int_M g(X,Z) ( d\iota(Y) \rho ) \wedge \rho\\
      &\qquad
      + \int_M ( d\iota(X)\rho ) \wedge ( \iota(Y)\rho )\wedge \iota(Z)g + \int_M g(Y,Z) ( d\iota(X)\rho ) \wedge\rho\\
      &\qquad
      + \int_M ( \iota(X)d\iota(Y)\rho ) \wedge \rho \wedge \iota(Z) g.
    \end{split}
  \end{equation*}
  Since
  \begin{equation*}
    \begin{split}
      \int_M X g(Y,Z) \dvol_\rho &= \int_M (\iota(X) d g(Y,Z)) \dvol_\rho\\
      &=
      \int_M (dg(Y,Z)) ( \iota(X) \rho ) \wedge \rho\\
      &=
      - \int_M g(Y,Z) ( d\iota(X) \rho ) \wedge \rho
    \end{split}
  \end{equation*}
  for all vector fields $X, Y,Z$ we find
  \begin{equation*}
    \begin{split}
      &\inner{\sigmahat}{-d\iota(\nabla_Z X + \nabla_X Z) \rho}_\rho\\
      &=
      \int_M ( d\iota(Z)\rho ) \wedge ( \iota(Y)\rho ) \wedge \iota(X) g
      - \int_M ( d\iota(Y) \rho ) \wedge ( \iota(Z) \rho ) \wedge \iota(X) g\\
      &\qquad
      + \int_M ( d\iota(X)\rho ) \wedge ( \iota(Y)\rho ) \wedge \iota(Z)g
      + \int_M ( \iota(X)d\iota(Y)\rho ) \wedge \rho \wedge \iota(Z) g\\
      &=
      \int_M B.
   \end{split}
  \end{equation*}
  Hence
  $$
  \inner{\sigmahat}{\frac12 d\iota(Z)\rhohat + \frac12 d\iota(X)\tauhat -
  \frac12 d\iota(\nabla_Z X + \nabla_XZ)\rho}_\rho = \frac12 \int_M A + B =
  \Gamma_{\sigmahat\tauhat\rhohat}.
  $$
  This proves the claim and the theorem.
\end{proof}

The following is an immidiate corollary.
\begin{corollary}[{\bf Geodesic Equation}]\label{cor:geodesic}
  The geodesic equation on the space $\sS_a$ with respect to the
  Donaldson metric is
  \begin{equation}\label{eq:geodesic}
    \frac{d^2}{dt^2} \rho_t = d\iota(X_t)d\iota(X_t)\rho_t +
    d\iota(\nabla_{X_t}X_t) \rho_t,
  \end{equation}
  where $X_t$ is the associated vector field of $\p_t \rho_t$.
  %The solution to
  %the geodesic equation with initial conditions $\rho_0$ and
  %$\left.\frac{d}{dt}\right|_{t=0} \rho_t$ is
  %$$
  %\rho_t = \exp(-t X_0)^* \rho_0.
  %$$
  %where $X_0$ is the associated vector field to
  %$\left.\frac{d}{dt}\right|_{t=0}
  %\rho_t$.
\end{corollary}

The next lemma gives an alternative forumula for the covariant derivative.
\begin{lemma}\label{lem:levicivita2}
  Let $\rho_t:\R \to \sS_a$ be a smooth path of symplectic forms with $\rho :=
  \rho_0$ and $\rhohat := \left.\p_t\right|_{t=0}\rho_t$. Let $X$ be the associated
  vector field of $\rhohat$. Let $Y_t: \R \to \Vect(M)$ be a smooth path of
  vector fields such that $*^\rho \iota(Y_t) \rho_t$ is exact and define
  \begin{equation*}
    \sigmahat_t := -d \iota(Y_t) \rho_t, \quad \sigmahat := \sigmahat_0,
    \quad Y := Y_0.
  \end{equation*}
  Then
  \begin{equation*}
    \nabla_{\rhohat}^\rho \sigmahat = - d\iota(\Yhat + \nabla_X Y)\rho, \qquad
    \Yhat := \left.\frac{d}{dt}\right|_{t=0} Y_t.
  \end{equation*}
\end{lemma}
\begin{proof}
  We have
  \begin{equation*}
    \begin{split}
      \nabla^\rho_\rhohat \sigmahat
      &=
      -\left. \frac{d}{dt} \right|_{t=0}
      d\iota(Y_t)\rho_t
      + \frac12 d\iota(Y)\rhohat + \frac12 d \iota(X)
      \sigmahat
      - \frac12 d \iota(\nabla_{Y}X + \nabla_X
      Y)\rho\\
      &=
      -d\iota(\Yhat)\rho - d\iota(Y)\rhohat + \frac12
      d\iota(Y)\rhohat
      + \frac12 d \iota(X) \sigmahat
      - \frac12 d \iota(\nabla_{Y}X + \nabla_X
      Y)\rho.
    \end{split}
  \end{equation*}
  Using the identity
  $
  \cL_{[X,Y]} = -[\cL_X,\cL_Y]$ and Cartan's formula for the Lie derivative we compute
  \begin{equation*}
    \begin{split}
      -2d\iota(Y)\rhohat
      + d\iota(Y)\rhohat &+ d \iota(X)
      \sigmahat
      - d \iota(\nabla_{Y}X + \nabla_X
      Y)\rho\\
      &\qquad=
      -d\iota(Y)\rhohat + d \iota(X)
      \sigmahat - d \iota(\nabla_{Y}X + \nabla_X
      Y)\rho\\
      &\qquad=
      \cL_{Y}\cL_X \rho - \cL_X
      \cL_{Y}\rho - d \iota(\nabla_{Y}X + \nabla_X
      Y)\rho\\
      &\qquad=
      -\cL_{[Y, X]}\rho - d \iota(\nabla_{Y}X + \nabla_X
      Y)\rho\\
      &\qquad=
      -2d\iota(\nabla_X{Y})\rho.
    \end{split}
  \end{equation*}
  Hence,
  \begin{equation*}
    \nabla_{\rhohat}^\rho \sigmahat = - d\iota(\Yhat + \nabla_X Y)\rho.
  \end{equation*}
  This proves the lemma.
\end{proof}

%We compute the formula for the Riemannian curvature tensor associated to the
%Levi-Civita connection of the Donaldson metric.
%\begin{corollary}[{\bf Curvature}]
  %Let $\rho_{s,t}:\R^2 \to \sS_a$ be a smooth 2-paramater family of symplectic
  %forms with $\rho := \rho_{0,0}$ and $\rhohat_1 :=
  %\left.\p_s\right|_{s=0}\rho_{s,0}$ and $\rhohat_2 :=
  %\left.\p_t\right|_{t=0}\rho_{0,t}$. Let $X_i$ be the associated vector field
  %of $\rhohat_i$. Let $Y_{s,t}: \R \to \Vect(M)$ be a smooth 2-parameter
  %family of vector fields such that $*^{\rho_{s,t}} \iota(Y_{s,t}) \rho_{s,t}$
  %is exact and define
  %\begin{equation*}
    %\sigmahat_{s,t} := -d \iota(Y_{s,t}) \rho_{s,t}, \quad \sigmahat :=
    %\sigmahat_{0,0},
    %\quad Y := Y_{0,0}.
  %\end{equation*}
  %Then,
  %\begin{equation*}
    %\nabla^\rho_{\rhohat_2}\nabla^\rho_{\rhohat_1} \sigmahat -
    %\nabla^\rho_{\rhohat_1}\nabla^\rho_{\rhohat_2} \sigmahat -
    %\nabla^\rho_{[\rhohat_1,\rhohat_2]}\sigmahat = -\iota(R(X_1, X_2) Y)\rho,
  %\end{equation*}
  %where $R$ denotes the Riemannian curvature tensor of the background metric
  %$g$.
%\end{corollary}
%\begin{proof}
  %By Lemma~\ref{lem:levicivita2}
  %$$
  %\nabla^\rho_{\rhohat_1} \sigmahat = -d\iota\left(\left.\frac{d}{ds}\right|_{s=0}
  %Y_{s,t} + \nabla_{X_1}Y_{0,t}\right)\rho_{0,t}
  %$$
  %TODO: We don't know if this is the associated vector field.
%\end{proof}

S.~Donaldson introduced the following energy functional on the space of
symplectic structures in a fixed cohomology class in \cite{DON1},
$$
\sE : \sS_a \to \R, \qquad
\sE(\rho) := \int_M \frac{2 \abs{\rho^+}^2}{\abs{\rho^+}^2 - \abs{\rho^-}^2}\dvol.
$$
The functional and the corresponding negative gradient flow with respect to
the Donaldson metric are further studied in \cite{KROMSAL} and \cite{KROM}.
It is shown in \cite{KROMSAL} that the gradient of $\sE$ with respect to the
Donaldson metric is the operator
\begin{align*}
  \grad\sE &: \sS_a \to T_\rho \sS_a\\
  \rho &\mapsto -d*^\rho d\Theta^\rho,
\end{align*}
where
\begin{align*}
  \Theta^\rho := * \frac{\rho}{u} - \frac12 \Abs{\frac{\rho}{u}}^2 \rho.
\end{align*}
We compute its associated vector field.
\begin{lemma}\label{lem:XgradE}
  The associated vector field $X_{\grad\sE}$ of $\grad\sE(\rho)$ is given by the
  two equivalent equations
  \begin{equation}\label{eq:XgradE}
    *^\rho d\Theta^\rho = \iota(X_{\grad\sE})\rho \iff d\Theta^\rho = \rho \wedge \iota\left( X_{\grad\sE} \right)g.
  \end{equation}
  In the hyperK\"ahler case,
  $$
  X_{\grad\sE} = -\sum_{i=1}^3 J_i X_{K_i},
  $$
  where
  $K_i :=
  \frac{\om_i\wedge\rho}{\dvol_\rho}$ and $X_{K_i}$ is the Hamiltonian vector
  field of $K_i$ with respect to the symplectic structure $\rho$.
\end{lemma}
\begin{proof}
  It is immediate that a vector field $X_{\grad\sE}$ defined by the first
  equation of~\eqref{eq:XgradE} satisfies the two
  conditions~\eqref{eq:associatedVecField} for $\rhohat = \grad\sE(\rho) =
  -d*^\rho d\Theta^\rho$. That the second equation is equivalent to the first
  follows from the identity
  $
  *^\rho \iota\left( X \right)\rho = -\rho \wedge \iota\left( X \right)g
  $
  proved in~\cite{KROMSAL}. In the hyperK\"ahler case it is shown in
  \cite{KROMSAL} that
  $
  d\Theta^\rho = *^\rho \sum_i \rho(J_i X_{K_i},\cdot)
  $.
  Hence it follows from the first equation in~\eqref{eq:XgradE} that
  $\rho(X_{\grad\sE},\cdot) = -\rho(\sum_i J_i X_{K_i}, \cdot)$. This proves the
  lemma.
\end{proof}

The Hessian operator of the energy functional $\sE$ is the operator $\cH: T_\rho
\sS_a \to T_\rho \sS_a$ defined by
$$
\cH_\rho \rhohat := \nabla_\rhohat^\rho \grad\sE(\rho).
$$
Associated to this operator is the Hessian quadratic form $\sH_\rho :T_\rho \sS_a \to
\R$ given by
$$
\sH_\rho (\rhohat) := \inner{\cH_\rho \rhohat}{\rhohat}_\rho.
$$
Since $\nabla^\rho$ is the Levi-Civita connection of the Donaldson metric, the
Hessian quadratic form equals $\left.\frac{d^2}{dt^2}\right|_{t=0} \sE(\rho_t)$
for a curve $\R \to \sS_a$ : $t \to \rho_t$ satisfying $\rho_0 = \rho$,
$\left.\frac{d}{dt}\right|_{t=0} = \rhohat$ and
$\left.\frac{d^2}{dt^2}\right|_{t=0}\rho_t = 0$.

\begin{theorem}[{\bf Covariant Hessian}]\label{thm:covhess}
  Let
  $
  \rho \in \sS_a
  $.
  Then the following holds.

  \smallskip\noindent{\bf (i)}
  The Hessian operator of the energy functional
  $
  \sE:\sS_a \to \R
  $
  is the linear operator
  \begin{equation}\label{def:hessop}
      \cH_\rho \rhohat
      =
      -d*^\rho d\Thetahat
      + d*^\rho
      \left(
      \rhohat \wedge \iota(X_{\grad\sE}) g
      \right)-
      d\iota\left( \nabla_X X_{\grad\sE} \right)\rho,
  \end{equation}
  where
  $
  \Thetahat := \frac{\rhohat + *^\rho \rhohat}{u} - \Abs{\frac{\rho^+}{u}}^2
  \rhohat,
  $
  $\nabla$ denotes the Levi-Civita connection of the metric
  $g$ and
  $
  X,X_{\grad\sE}
  $
  are the associated vector fields to $\rhohat$ respectively $\grad\sE(\rho)$.

  \smallskip\noindent{\bf (ii)}
  The Hessian of $\sE$ is the quadratic form
  \begin{equation}
      \sH_\rho(\rhohat)
      :=
      \int_M \Thetahat \wedge
      \rhohat
      + \int_M \left( \iota(X)\rhohat - \iota(\nabla_X X)\rho
      \right)\wedge *^\rho \iota(X_{\grad\sE})\rho.
  \end{equation}

  \smallskip\noindent{\bf (iii)}
  In the hyperK\"ahler case the Hessian of $\sE$ is given by
  \begin{equation}
    \sH_\rho(\rhohat) = \int_M \sum_i\left( \Hhat_i^2 \dvol_\rho + \om_i
    (X,\nabla_{X_{K_i}}X)\right)\dvol_\rho,
  \end{equation}
  where
  $
    \Hhat_i := \frac{(d\iota(X)\om_i) \wedge \rho}{\dvol_\rho}
  $
  and
  $
    K_i := \frac{\om_i \wedge \rho}{\dvol_\rho}.
  $
\end{theorem}
\begin{proof}
  We prove (i). Let $X$ and $X_{\grad\sE}$ be the associated vectorfields
  of $\rhohat$ and $\grad\sE$. Let $\rho_t:\R \to \sS_a$ be a path of symplectic
  forms such that $\left. \frac{d}{dt}\right|_{t=0}\rho_t = \rhohat$. By
  Lemma~\ref{lem:levicivita2}
  $$
  \nabla_{\rhohat} \grad \sE(\rho) =  -\iota(\Xhat_{\grad \sE} + \nabla_X X_{\grad
  \sE})\rho
  $$
  where $\Xhat_{\grad \sE} = \left. \frac{d}{dt}\right|_{t=0}X_{\grad \sE}$.
  By Lemma~\ref{lem:XgradE} we have $d\Theta^\rho = \rho\wedge \iota\left(
  X_{\grad\sE} \right)g$ and hence
  $$
  d\Thetahat = \rhohat \wedge
  \iota(X_{\grad\sE})g + \rho \wedge \iota\left( \Xhat_{\grad\sE} \right)g,
  $$
  where
  $
  \Thetahat = \left.\frac{d}{dt}\right|_{t=0}\Theta^{\rho_t}.
  $
  It follows that
  \begin{equation*}
      \iota(\Xhat_{\grad\sE})\rho
      =
      *^\rho\left( \rho \wedge \iota\left( \Xhat_{\grad\sE} \right)g \right)\\
      =
      *^\rho d \Thetahat -
      *^\rho\left( \rhohat \wedge\iota\left( X_{\grad\sE} \right)g \right).
  \end{equation*}
  Hence,
  \begin{equation*}
      \nabla_\rhohat \grad\sE\left( \rho \right)
      =
      - d *^\rho d
      \Thetahat
      + d*^\rho\left( \rhohat \wedge \iota(X_{\grad\sE})g \right) -
      d\iota\left( \nabla_X X_{\grad\sE} \right)\rho.
  \end{equation*}
  That
  $
  \Thetahat = \frac{\rhohat + *^\rho \rhohat}{u} -
  \Abs{\frac{\rho^+}{u}}\rhohat
  $
  is proved in \cite{KROMSAL}. This proves (i).

  We prove (ii). By part (i)
  \begin{equation*}
    \begin{split}
      \sH_\rho(\rhohat)
      &=
      \inner{\cH_\rho(\rhohat)}{\rhohat}_\rho\\
      &=
      \inner{-d *^\rho
        d\left( \frac{\rhohat + *^\rho \rhohat}{u} - \Abs{\frac{\rho^+}{u}}^2
      \rhohat \right)}{\rhohat}_\rho\\
      &\qquad
      + \inner{d*^\rho\left( \rhohat \wedge
        \iota(X_{\grad\sE})g \right)}{\rhohat}_\rho  + \inner {-
          d\iota\left( \nabla_X X_{\grad\sE} \right)\rho}{\rhohat}_\rho\\
      &=:
      A + B + C.
    \end{split}
  \end{equation*}
  By the definition of the Donaldson metric
  \begin{equation*}
    \begin{split}
    A
    &=
    \int_M\left( *^\rho d\left( \frac{\rhohat + *^\rho \rhohat}{u} - \Abs{\frac{\rho^+}{u}}^2
      \rhohat \right) \right)\wedge *^\rho \iota(X) \rho\\
    &=
    \int_M\left( \frac{\rhohat + *^\rho \rhohat}{u} - \Abs{\frac{\rho^+}{u}}^2
      \rhohat \right) \wedge \left(- d\iota(X) \rho\right)\\
    &=
    \int_M\left( \frac{\rhohat + *^\rho \rhohat}{u} - \Abs{\frac{\rho^+}{u}}^2
      \rhohat \right) \wedge \rhohat.
    \end{split}
  \end{equation*}
  Likewise,
  \begin{equation*}
    \begin{split}
      B
      &=
      -\int_M\rhohat\wedge \left(\iota(X_{\grad\sE})g\right) \wedge \iota(X) \rho\\
      &=
      -\int_M \left(\iota(X) \rhohat\right) \wedge \left(\iota\left( X_{\grad\sE}
      \right)g\right) \wedge \rho - \int_M g(X_{\grad\sE}, X)\, \rhohat \wedge
      \rho\\
      &=
      \int_M \left(\iota(X) \rhohat\right) \wedge *^\rho \left(\iota\left( X_{\grad\sE}
      \right)\rho\right) - \int_M g(X_{\grad\sE}, X)\, \rhohat \wedge
      \rho.
    \end{split}
  \end{equation*}
  Since $\inner{-d\iota(X)\rho}{-d\iota(Y)\rho}_\rho = \int_M (\iota(X) \rho)
  \wedge *^\rho \iota(Y) \rho = \int_M g(X,Y) \dvol_\rho$
  for $X$ associated to $\rhohat$ and $Y$ an arbitrary vector field we have
  \begin{equation*}
    \begin{split}
      C
      &=
      \int_M g(\nabla_X X_{\grad\sE}, X) \dvol_\rho\\
      &=
      \int_M\left(\iota(X)d g(X_{\grad\sE},X) - g(X_{\grad\sE}, \nabla_X X)
      \right)\dvol_\rho\\
      &=
      \int_M d g(X_{\grad\sE},X) \wedge \left(\iota(X) \rho\right) \wedge \rho  - \int_M g(X_{\grad\sE}, \nabla_X X)
      \dvol_\rho\\
      &=
      \int_M g(X_{\grad\sE},X)  \, \rhohat \wedge \rho   - \int_M \iota(\nabla_X
      X)\rho \wedge *^\rho \iota(X_{\grad\sE})\rho.
    \end{split}
  \end{equation*}
  Hence,
  \begin{equation*}
    \begin{split}
        A + B + C &=
            \int_M \left( \frac{\rhohat + *^\rho
            \rhohat}{u} - \Abs{\frac{\rho^+}{u}}^2 \rhohat \right) \wedge
            \rhohat\\
            &\qquad
            + \int_M \left( \iota(X)\rhohat - \iota(\nabla_X X)\rho
            \right)\wedge *^\rho \iota(X_{\grad\sE})\rho.
    \end{split}
  \end{equation*}
  This proves (iii).

  We prove (iii). Assume the hyperK\"ahler case. The following identities are proved in
  $\cite{KROMSAL}$,
  \begin{align*}
    \grad\sE(\rho)
    &=
    d \sum_i^{3} dK_i \circ J_i^\rho\\
    \Thetahat
    &= \sum_i^3 \Khat_i^2\om_i^\rho - \frac{1}{2}\sum_i^3K_i^2 \rhohat
    \\
    \int_M \Thetahat \wedge \rhohat
    &=
    \int_M \sum_i \left( \Khat_i^2 \dvol_\rho
    -\frac{1}{2 } K_i^2 \rhohat \wedge \rhohat\right),
  \end{align*}
  where $\Khat_i = \frac{\om_i^\rho \wedge \rhohat }{\dvol_\rho}$,
  $\rho(J_i^\rho\cdot, \cdot) := \rho(\cdot, J_i \cdot)$ and $\om_i^\rho =
  \om_i - K_i \rho$. From (i) we have
  \begin{align*}
    \sH_\rho(\rhohat)
    &=
    \int_M \Thetahat \wedge \rhohat - \int_M \rhohat \wedge
    \iota\left( X_{\grad\sE} \right)g \wedge\iota(X) \rho\\
    &\qquad + \int_M g\left(
    \nabla_X X_{\grad\sE}, X \right) \dvol_\rho\\
    &=: \int_M \Thetahat \wedge \rhohat + D + E.
  \end{align*}
  By Lemma~\ref{lem:XgradE} we have $X_{\grad\sE} = - \sum_i J_i X_{K_i}$.
  Therefore
  \begin{equation*}
    D
    =
    - \int_M \rhohat \wedge
    \iota\left( X_{\grad\sE} \right)g \wedge\iota(X) \rho
    =
    \int_M \sum_i \iota\left( X_{K_i} \right)\om_i\wedge \iota\left( X
    \right)\rho \wedge \rhohat
  \end{equation*}
  and
  \begin{equation*}
    E
    =
    \int_M g\left(
    \nabla_X X_{\grad\sE}, X \right) \dvol_\rho
    =
    \int_M \sum_i \om_i\left(X, \nabla_X X_{K_i}\right) \dvol_\rho.
  \end{equation*}
  It now follows from Lemma~4.3 in \cite{KROMSAL} that
  \begin{align*}
    \sH_\rho (\rhohat)
    &=
    \int_M \sum_i \left( \Khat_i^2 \dvol_\rho
      -\frac{1}{2 } K_i^2 \rhohat \wedge \rhohat\right) + D + E\\
    &=
    \int_M \sum_i \left( \Hhat_i^2\dvol_\rho + \om_i\left( X,
    \nabla_{X_{K_i}}X \right) \right)\dvol_\rho.
  \end{align*}
  This proves (iii) and the theorem.
\end{proof}

\begin{remark}
  The Hessian operator $\cH: T_\rho \sS_a \to T_\rho \sS_a$ given
  by~\eqref{def:hessop} is a non-local differential operator of degree two. It
  is non-local because of the last term $d \iota\left(\nabla_X
  X_{\grad\sE}\right)\rho$, which involves solving the equation
  $$
  -d \iota(X) \rho = \rhohat, \qquad *^\rho \iota(X) \rho \in \im d
  $$
  for the associated vector field $X$ of $\rhohat$. Its leading term
  \begin{equation*}
    \begin{split}
      - d*^\rho d \Thetahat
      &=
      -d *^\rho d \left( \frac{\rhohat + *^\rho
      \rhohat}{u} - \Abs{\frac{\rho^+}{u}}^2 \rhohat \right)\\
      &= -2d\frac{*^\rho}{u} d
      \rhohat^{+^\rho} + d *^\rho \left(
      \frac{d u}{u^2} \wedge (\rhohat + *^\rho \rhohat)
      \right) + d *^\rho \left(d \Abs{\frac{\rho^+}{u}}^2 \wedge
      \rhohat\right)\\
      &=
      (d^{*^\rho} \frac{1}{u}d + d\frac{1}{u}d^{*\rho})\rhohat + d *^\rho \left(
      \frac{d u}{u^2} \wedge (\rhohat + *^\rho \rhohat)
      \right)
      \\
      &\qquad \qquad
      + d *^\rho \left(d \Abs{\frac{\rho^+}{u}}^2 \wedge
      \rhohat\right)
    \end{split}
  \end{equation*}
  is an elliptic differential operator.
\end{remark}

%%%%%%%%%%%%%%%%%%%%%%%%%%%%%%%%%%%%%%%%%%%%%%%%%%
%%%%%%%%%%%%%%%%%%%%%%%%%%%%%%%%%%%%%%%%%%%%%%%%%%
%%%%%%%%%%%%%%%%%%%%%%%%%%%%%%%%%%%%%%%%%%%%%%%%%%
%%%%%%%%%%%%%%%%%%%%%%  References %%%%%%%%%%%%%%%
%%%%%%%%%%%%%%%%%%%%%%%%%%%%%%%%%%%%%%%%%%%%%%%%%%
%%%%%%%%%%%%%%%%%%%%%%%%%%%%%%%%%%%%%%%%%%%%%%%%%%
%%%%%%%%%%%%%%%%%%%%%%%%%%%%%%%%%%%%%%%%%%%%%%%%%%

%\newpage
%\addcontentsline{toc}{part}{References}

\end{document}